\documentclass{amsart}
\usepackage{amsmath}
\usepackage{amsthm}
\usepackage{amssymb}
\usepackage{color}
\usepackage{enumerate}

\newtheorem{thm}{Theorem}
\newtheorem{lem}[thm]{Lemma}

\newtheorem{prop}[thm]{Proposition}
\newtheorem{rk}[thm]{Remark}

\newcommand{\R}{{\mathbb{R}}}
\newcommand{\N}{{\mathbb{N}}}
\newcommand{\Z}{{\mathbb{Z}}}
\newcommand{\cF}{{\mathcal{F}}}
\newcommand{\cC}{{\mathcal{C}}}
\newcommand{\cZ}{{\mathcal{Z}}}
\newcommand{\cL}{{\mathcal{L}}}
\newcommand{\e}{\epsilon}
\newcommand{\vip}{\vskip.15cm}
\newcommand{\indiq}{{\bf 1}}
\newcommand{\E}{\mathbb{E}}
\newcommand{\PR}{\mathbb{P}}
\newcommand{\intot}{\int_0^t }
\newcommand{\dd}{{\rm d}}
\newcommand{\ddiv}{{\rm div}}
\newcommand{\tV}{{\tilde V}}

\newcommand{\tX}{{\tilde X}}
\newcommand{\tB}{{\tilde B}}

\newcommand{\bu}{{\mathbf{u}}}

\begin{document}

\title{Simulated annealing in $\R^d$ with slowly growing potentials}

\author{Nicolas Fournier, Pierre Monmarch\'e and Camille Tardif}

\address{Nicolas Fournier and Camille Tardif : Sorbonne Universit\'e, 
LPSM-UMR 8001,
Case courrier 158, 75252 Paris Cedex 05, France. 
{\tt nicolas.fournier@sorbonne-universite.fr},
{\tt camille.tardif@sorbonne-universite.fr}.}

\address{Pierre Monmarch\'e : Sorbonne Universit\'e, LJLL-UMR 7598,
Case courrier 187, 75252 Paris Cedex 05, France. 
{\tt pierre.monmarche@sorbonne-universite.fr}.}




\begin{abstract}
We use a localization procedure to weaken the growth assumptions 
of Royer \cite{r}, Miclo \cite{mic} and Zitt \cite{z}
concerning the continuous-time simulated annealing in $\R^d$.
We show that a transition occurs for potentials growing like
$a \log \log |x|$ at infinity. We also study a class of potentials
with possibly unbounded sets of local minima.
\end{abstract}

\maketitle

\section{Introduction and results}

\subsection{Notation and main result}

We adopt, in the whole paper, the following setting.

\vip

\noindent {\bf Assumption} $(A)$. {\it We work in dimension $d\geq 2$. 
The function $U:\R^d \to \R_+$ is of class $C^\infty$, 
satisfies $\lim_{|x|\to \infty} U(x)=\infty$, $\min_{x\in \R^d} U(x)=0$. 
For $x,y\in \R^d$, we set 
$$
E(x,y)=\inf \Big\{\max_{t\in [0,1]}U(\gamma(t))-U(x)-U(y) \;:\; 
\gamma \in C([0,1],\R^d), \gamma(0)=x,\gamma(1)=y \Big\}
$$
and we suppose that $c_*=\sup\{E(x,y) : x,y\in\R^d\}<\infty$.}

\vip

Actually,  $c_*=\sup\{E(x,y) : x$ local minimum of $U$, $y$ global minimum of $U\}$
and
represents the maximum energy required to reach a global minimum $y$ when starting 
from anywhere else.

\vip

We fix $x_0 \in \R^d$, $c>0$ and $\beta_0\geq 0$
and consider the time-inhomogeneous S.D.E.
\begin{equation}\label{eds}
X_t = x_0 + B_t - \frac 12\intot \beta_s \nabla U(X_s) \dd s \quad \hbox{where} 
\quad \beta_t=\beta_0+\frac{\log (1+t)}{c}.
\end{equation}
Here $(B_t)_{t\geq0}$ is a $d$-dimensional Brownian motion.
For $R>0$, we set $B(R)=\{x\in \R^d : |x| < R\}$.
We will work under one of the three following conditions.

\vip

\noindent {\bf Assumption} $(H_1(a))$. {\it There is $A_0\geq 2$ such that
$x \cdot \nabla U(x) \geq a/\log|x|$ for all $x\in \R^d \setminus B(A_0)$.}

\vip

\noindent {\bf Assumption} $(H_2(\alpha))$. {\it There are $\delta_0>0$ and three
sequences $(a_i)_{i\geq 1}$, $(b_i)_{i\geq 1}$ 
and $(\delta_i)_{i\geq 1}$ such that $0\leq a_1<b_1\leq a_2<b_2\leq \dots$
and, for all $i\geq 1$, $\delta_i \geq \delta_0$, $b_i \geq a_i+ \alpha \delta_i$,  and}
$$
|x| \in [a_i,b_i] \quad \Longrightarrow \quad 
\frac{x}{|x|}\cdot \nabla U(x) \geq \frac 1{\delta_i}.
$$

We say that a set $\cZ\subset \R^d$ is a {\it ring} if
it is $C^\infty$-diffeomorphic to $\cC=\{x\in \R^d : |x| \in (1,2)\}$. A ring $\cZ$
is connected, open, bounded and $\R^d\setminus\cZ$
has precisely two connected components, one begin bounded 
(denoted by $\cZ^-$), 
the other one being unbounded (denoted by $\cZ^+$).

\vip

\noindent {\bf Assumption} $(H_3(\alpha,\xi))$. {\it There are $\e>0$, three sequences
$(u_i)_{i\geq 1}$, $(v_i)_{i\geq 1}$ and $(\kappa_i)_{i\geq 1}$ and 
a family of rings $\{\cZ_i : i\geq 1\}$
such that $\cup_{i\geq 1} \cZ_i^-=\R^d$ and 
for all $i\geq 1$,  $(\cZ_i^+)^c \subset \cZ^-_{i+1}$, 
$v_i \geq u_i + \alpha \max\{1,\e\kappa_i\}$, 
$\partial \cZ_i^- \subset \{x\in\R^d : U(x)=u_i\}$, 
$\partial \cZ_i^+ \subset \{x\in\R^d : U(x)=v_i\}$ and}
$$
x\in \bar\cZ_i \quad \Longrightarrow \quad  
|\nabla U(x)| \in (0,\kappa_i]\;\; \hbox{ {\it and} }\;\; 
\frac{\Delta U(x)}{|\nabla U(x)|^2} \in (-\infty, \xi].
$$  

Our main result is as follows.

\begin{thm}\label{mr}
Assume $(A)$ and fix $c>c_*$ and $\beta_0\geq 0$.
Assume either $(H_1(a))$ for some $a>c(d-2)/2$ or $(H_2(\alpha))$ for some $\alpha>c$
or $(H_3(\alpha,\beta_0))$ for some $\alpha>c$.
The S.D.E. \eqref{eds} has a pathwise unique solution 
$(X_t)_{t\geq 0}$ and 
$U(X_t)$ tends to $0$, in probability, as $t\to \infty$.
\end{thm}

It is well-known that, even with a fast growing potential, the condition $c>c_*$ 
is necessary, see Holley-Kusuoka-Stroock \cite[Corollary 3.11]{hks} 
for the case where $\R^d$ 
is replaced by a compact manifold. 
The following example shows that in some sense, $(H_1(a))$ is sharp.

\begin{prop}\label{contrex}
Assume that $d\geq 3$. Fix $\beta_0=0$, $c>0$ and $a\in (0,c(d-2)/2)$.
For $\alpha \in (a,c(d-2)/2)$, set $U(x)=\alpha \log(1+\log(1+|x|^2))$, which satisfies $(A)$ 
with $c_*=0$ and $(H_1(a))$. 
For any $x_0\in\mathbb R^d$, the solution $(X_t)_{t\geq 0}$ to \eqref{eds} 
satisfies $\PR(\lim_{t\to \infty}U(X_t)=\infty)>0$.
\end{prop}

The next example shows that one can build some 
{\it oscillating} potentials, growing more or
less as slow as one wants, such that Theorem \ref{mr} applies. Hence in some sense,
$(H_1(a))$ is far from being satisfying.

\begin{prop}\label{contrex2}
Fix $d\geq 2$ and $p\geq 1$. We can find $U$ satisfying $(A)$ with $c_*=1$ 
and $(H_2(2))$ such
that $\log^{\circ p} |x| \leq  U(x) \leq 3 \log^{\circ p} |x|$ outside a compact.
Theorem \ref{mr} applies when $c\in (1,2)$.
\end{prop}

\subsection{Motivation and bibliography}

The problem under consideration, called {\it simulated annealing}, 
has a long history, see the introduction of Zitt \cite{z}. The goal is to find
numerically a global minimum of a given function $U:\R^d\to \R$, by using a gradient 
approach, perturbed by a stochastic noise.
One thus considers the S.D.E. $\dd Y_t = \sqrt{\sigma_t}\dd B_t -\frac12\nabla U(Y_t)\dd t$.
The noise intensity $\sigma_t$ has to be small, so that there is some hope to
spend most of the time close to a global minimum, but large enough so that
one is sure not to remain stuck close to a local minimum.
Changing time, one can equivalently study $(Y_t)_{t\geq 0}$ or the solution
$(X_t=Y_{\rho_t})_{t\geq 0}$ to \eqref{eds} with $\beta_t=1/\sigma_{\rho_t}$, where $(\rho_t)_{t\geq 0}$
is the inverse of $(\int_0^t \sigma_s \dd s)_{t\geq 0}$. The important point is that 
for $c>0$ fixed, as $t\to\infty$,  $\beta_t \sim c^{-1}\log t$ if and only if $\sigma_t \sim c(\log t)^{-1}$.
In each of the the references
cited below, one choice or the other is used.

\vip

After a first partial result by Chiang-Hwang-Sheu \cite{chs}, this question
has been solved by Royer \cite{r} and Miclo \cite{mic} when assuming that $U$ grows
sufficiently fast at infinity, always assuming at least that 
\begin{equation}\label{bg}
\lim_{|x|\to \infty} U(x)= \lim_{|x|\to \infty} |\nabla U(x)|=\infty \quad \hbox{and}
\quad \forall \; x\in\R^d, \quad \Delta U(x) \leq C + |\nabla U(x)|^2
\end{equation}
for some constant $C>0$. The case where $\R^d$
is replaced by a compact Riemannian manifold was solved by 
Holley-Kusuoka-Stroock \cite{hs,hks}.
All these studies deeply rely on some 
Poincar\'e and log-Sobolev inequalities that require, in the non-compact case,
some conditions like \eqref{bg}.

\vip

These conditions \eqref{bg} imply that all the local minima of $U$ are
located in a compact set. Also, if  $U$ behaves like $U(x)=|x|^r$
for some some $r>0$ outside a compact, then  \eqref{bg} holds true if and only if
$r>1$.
In \cite{z}, Zitt weakens the condition \eqref{bg}, 
using similar (but more involved) functional analysis methods, relying on 
some weak Poincar\'e inequalities.
However, many technical conditions are still assumed, which in particular
imply that all the local minima of $U$ are
located in a compact set, and that $U(x) \geq [\log |x|]^{r}$ outside a compact,
for some $r>1$.

\vip

The questions we address in this paper are thus the following. First, can one
find the {\it minimum} growth rate required for the simulated annealing to be successful ?
Second, can we allow for some potentials with unbounded set of local minima ?
We give answers to these questions, thanks to a localization procedure,
using as a {\it black box} the results of \cite{hks} in the compact case.

\subsection{Comments on the assumptions}
We could probably treat the case where $d=1$, but
some local times would appear here and there, this would change the definition of
rings, etc. Also, $(H_1(a))$ might be weakened in dimension $2$, as is rather clear
from Theorem~\ref{mr}, since we assume that $a>c(d-2)/2$. This is due
to the fact that the Brownian motion is recurrent in dimension $2$.
To simplify the presentation as much as possible, we decided not to address
these problems.

\vip

Assumption $(H_1)$ is rather clear and allows for very slowly growing potentials.
Any potential $U:\R^d\to \R$ satisfying, outside a compact,
$U(x)=|x|^r$ or $U(x)=(\log |x|)^r$, with $r>0$, satisfies $(H_1(a))$
for all $a>0$. And, of course, if $U(x)=a\log \log |x|$ outside a compact,
$(H_1(a))$ is satisfied. Proposition \ref{contrex} shows that in some loose sense,
the condition $(H_1(a))$ with $a>c(d-2)/2$ is optimal.
Observe also, and this is rather surprising, that $(H_1(a))$ does not guarantee at
all that the invariant measure $\exp(-\beta U(x))\dd x$ 
of the S.D.E. $\dd X^\beta_t=\dd B_t - \frac{\beta}2
\nabla U(X^\beta_t)\dd t$ with $\beta>0$ fixed, even large, can be normalized as a
probability measure.

\vip

We tried a lot to replace $(H_1(a))$ by its {\it integrated} version
$U(x) \geq a \log \log |x|$ outside a compact, and we did not succeed at all, 
even with the idea to get a much less sharp condition. This integrated condition
would be much more satisfactory, in particular since it would allow for 
potentials with unbounded sets of local minima.

\vip

Assumption $(H_2)$ is less clear, and might be improved, although we
tried to be as optimal as possible. The main idea is that a potential $U$
satisfies $(H_2(\alpha))$ if there are infinitely many {\it annuli} on which
$U$ increases at least of $\alpha$, sufficiently uniformly. Between these annuli,
the potential can behave as it wants, and in particular it may have
many local minima. Observe that $(H_2(\alpha))$ does not imply that
$\lim_{|x|\to \infty} U(x) = \infty$. However, one easily gets convinced
that $(H_2(\alpha))$, together with the condition $\alpha>c_*$,
implies that $\lim_{|x|\to \infty} U(x) = \infty$.

\vip

Assumption $(H_3)$ resembles much $(H_2)$. It is less general in that some
conditions on $\Delta U$ are imposed, but more general in that a ring allows for much
more general shapes than an annulus. Much less radial symmetry
is assumed.

\vip

Finally, $(H_2)$ and $(H_3)$ are not strictly more general 
than $(H_1)$. They are more intricate and thus harder to optimize. The following examples, 
that illustrate this fact, are 
not very interesting from the point of view of $(H_2)$ and $(H_3)$, since
the potentials below are radially symmetric and increasing, but they give an idea of the possibilities.

\vip 

$\bullet$ If $U(x)=(\log |x|)^{r}$, outside a compact,
with $r\in (0,1)$, then $U$ satisfies $(H_1(a))$ for all $a>0$. But it does not satisfy 
$(H_2(\alpha))$ for any $\alpha>0$, because we would have, for $i$ large enough, 
$\delta_i\geq b_i (\log b_i)^{1-r} / r$, whence $b_i\geq a_i + b_i (\log b_i)^{1-r} /r$. This is not possible,
since $b_i$ must increase to $\infty$ as $i\to \infty$. 
If $d\geq 3$, neither does it satisfy 
$(H_3(\alpha,\xi))$ for any $\alpha>0$ and $\xi>0$, since
$\lim_{|x|\to \infty} |\nabla U(x)|^{-2}\Delta U(x)= \infty$.

\vip

$\bullet$ If $U(x)=\kappa \log |x|$, outside a compact, with $\kappa>0$,
then $(H_1(a))$ is satisfied for all
$a>0$. Next, $(H_2(\alpha))$ 
is fulfilled if $\kappa>\alpha$: choose, for $i$ large enough, 
$a_i=q^{i}$, $b_i=a_{i+1}$, $\delta_i=q^{i+1}/\kappa$, with 
$q>1$ such that $q\geq 1+\alpha q/\kappa$. Finally, $(H_3(\alpha,\xi))$ is met if $\alpha>0$ and
$\xi\geq (d-2)/\kappa$: choose, for $i$ large enough, 
$\cZ_i=B(\exp(v_i/\kappa))\setminus B(\exp(u_i/\kappa))$ with $u_i=i \alpha$ and $v_i=u_{i+1}$ and $\kappa_i=1$.

\vip

$\bullet$ If $U(x)=(\log |x|)^r$, outside a compact,
with $r>1$, then $U$ satisfies $(H_1(a))$, $(H_2(\alpha))$ and $(H_3(\alpha,\xi))$ for all $a>0$, 
$\alpha>0$, $\xi>0$. For example, $(H_2(\alpha))$ is satisfied with, for $i$ large enough, 
$a_i=2^i$, $b_i=a_{i+1}$ and $\delta_i=b_i/(2\alpha)$.

\vip

As a conclusion, although we found some new results, the situation remains rather unclear.

\subsection{Main ideas of the proof}

Assume $(A)$ and fix $c>c_*$. First, it is rather natural to deduce 
the two following points
from the compact case \cite{hks}.

\vip

(a) Under the condition, to be verified, that $\sup_{t\geq 0} |X_t|<\infty$ a.s.,
then $U(X_t)\to 0$ in probability.

\vip

(b) If $G$ is an open connected set containing $x_0$ and the global minima of $U$
and such that $\partial G \subset \{x \in \R^d : U(x)\geq\alpha\}$ for some $\alpha>c$,
then $\PR(\forall \;t\geq0,\; X_t \in G)>0$.

\vip

The proof under $(H_1)$ then follows from two main arguments. First, a careful 
comparison of $(|X_t|)_{t\geq 0}$ with some Bessel process shows that $X$ cannot tend to
infinity, and thus visits infinitely often a compact set.
Second, each time it visits this compact set, it may remain stuck forever in it
with positive probability 
by point (b). 
With some work, we bound from below uniformly this probability.
Hence the process is eventually stuck in this compact set, so that we can apply (a).

\vip

The proof under $(H_2)$ or $(H_3)$ is rather easier. On the event where 
$\sup_{t\geq 0} |X_t| =\infty$, the process must cross all the annuli (or rings) 
in which $U$ is supposed to be {\it sufficiently increasing}.
But using some comparison arguments and point (b) above, there is a positive
probability that the process does not manage to cross a given annulus. Here again,
there is some work to get some uniform lowerbound. At the end, the process
can cross only a finite number of annuli (or rings), so that we can apply (a).

\subsection{Plan of the paper}
In the next section, we recall some results of Holley-Kusuoka-Stroock \cite{hks}
and deduce points (a) and (b) mentioned in the previous subsection.
We finally recall some classical facts about Bessel
processes. The other sections can be read independently. Sections
\ref{hh1}, \ref{hh2} and \ref{hh3} are respectively devoted
to the proofs of Theorem \ref{mr} under $(H_1)$, $(H_2)$ and $(H_3)$. We conclude the paper
with Section \ref{contrexxx} which contains the proofs of Propositions \ref{contrex}
and \ref{contrex2}.

\vip

As a final comment, let us mention that we use many similar comparison arguments.
We gave up producing a unified lemma, because it rather complicates the presentation,
since the time-life of the processes vary, etc, and because each time, 
the proof is very quick.

\section{Preliminaries}
We first recall some results of Holley-Kusuoka-Stroock on which our study entirely relies.
Recall that the constant $c_*$, concerning $U$, was introduced in Assumption $(A)$.
When considering a similar constant for another potential, we indicate it in superscript.

\begin{thm}[{\cite[Theorem 2.7 and Lemma 3.5]{hks}}]\label{HKS}
Consider a compact connected finite-dimensional Riemannian manifold $M$, as well as a 
$C^\infty$ function $V:M\to \R_+$
satisfying $\min_M V = 0$. We introduce $c_{*}^{V}=\sup\{E_V(x,y) : x,y\in M\}$, where
$$
E_V(x,y)=\inf \Big\{\max_{t\in [0,1]}V(\gamma(t))-V(x)-V(y) \;:\;
\gamma \in C([0,1],M), \gamma(0)=x,\gamma(1)=y \Big\}.
$$
Consider $c>c_{*}^{V}$ and $\beta_0\geq 0$, set $\beta_t=\beta_0+ [\log(1+t)]/c$ 
and consider the inhomogeneous 
$M$-valued 
diffusion $(Y_t)_{t\geq 0}$ with generator $\cL_{t}\phi(y)=\frac12\ddiv[\nabla \phi](y)
-\frac12\beta_t \nabla V(y)\cdot\nabla
\phi(y)$, for $y \in M$ and $\phi \in C^\infty (M)$, starting from some $y_0 \in M$. 
We denoted by
$\ddiv$ and $\nabla$ the Riemannian divergence and gradient operators. 

\vip

(i) It holds that $V(Y_t)\to 0$ in probability as $t\to \infty$.

\vip

(ii) Fix $\alpha \in (c_*^{V},c)$ and consider a connected open subset $G$ 
of $M$ satisfying the conditions that
$\{x\in M : V(x)=0\}\subset G$ and $\partial G \subset \{x\in M : V(x)\geq \alpha\}$.
If $y_0 \in G$, then $\PR(\forall \; t\geq 0 , \; Y_t \in G)>0$.
\end{thm}

Actually, only the case where $\beta_0=0$ is treated in \cite{hks}, but this 
is not an issue.
Under $(A)$, $\nabla U$ is locally lipschitz continuous, whence the following observation.

\begin{rk}\label{nn}
Assume $(A)$. The equation
\eqref{eds} has a pathwise unique maximal solution $(X_t)_{t\in [0,\zeta)}$,  where
$\zeta$ takes values in $(0,\infty)\cup\{\infty\}$ and with 
$\PR(\{\zeta = \infty\}\cup\{\zeta<\infty, \lim_{t \uparrow \zeta} |X_t|=\infty\})=1$.
\end{rk}

We now show how the above results of \cite{hks} may extend to 
the non-compact case.

\begin{lem}\label{lemzero}
Assume $(A)$, fix $c>c_*$ and $\beta_0\geq 0$, and consider $(X_t)_{t\in [0,\zeta)}$ 
as in Remark \ref{nn}.

\vip

(i) Fix $\alpha \in (c_*,c)$ and consider a bounded connected open subset $G$ 
of $\R^d$ such that
$x_0\in G$, 
$\{x\in \R^d : U(x)=0\}\subset G$ and $\partial G \subset \{x\in \R^d : 
U(x)\geq \alpha\}$.
Then $\PR(\zeta=\infty$ and $\forall \; t\geq 0 , \; X_t \in G)>0$.

\vip

(ii) Assume that $\PR(\zeta=\infty$ and $\sup_{t\geq 0} |X_t| <\infty)=1$.
Then $U(X_t)\to 0$ in probability as $t\to \infty$.
\end{lem}

\begin{proof}
For $R>0$, we introduce the flat torus $M_R= [-R,R)^d$, i.e. $\R^d$ quotiented by the 
equivalence relation $x\sim y$ if and only if $(x_i-y_i)/(2R) \in \Z$ for all 
$i=1,\dots,d$.

\vip

We also fix $c>c_*$ and $\alpha \in (c_*,c)$ for the whole proof.

\vip

{\it Step 1.} For all $A \geq 1$, there exist $R_A>A$
and a $C^\infty$ function $V_A:M_{R_A} \to \R_+$
such that $c_*^{V_A}<\alpha$, $\min_{M_{R_A}} V_A =0$, $\{x \in M_{R_A} : V_A(x)=0\}= \{x \in \R^d : U(x)=0\}$ 
and $U(x)=V_A(x)$ for all $x\in B(A)$.

\vip

Indeed, let $m_A=\max_{B(A)}U +1$, 
and $D_A=\{x\in \R^d : U(x)\leq m_A\}$, which is compact, since $U$ is continuous 
and satisfies $\lim_{|x|\to \infty} U(x) = \infty$. 
Hence there is $R_A>A$ such that $D_A \subset [-(R_A-1),(R_A-1)]^d$. 
We then introduce the continuous map $\tV_A:M_{R_A}\to \R_+$ defined by
$$
\tV_A(x)= \min\{U(x),m_A\}=U(x) \indiq_{\{x \in D_A\}} + m_A \indiq_{\{x \in M_{R_A}\setminus D_A\}}.
$$
Since $\tV_A$ is constant outside $D_A$ and since $\tV_A=U$ on $D_A$,
one easily checks that $c_*^{\tV_A}\leq c_*$.
\vip
We next consider $V_A: M_{R_A} \to \R_+$ of class $C^\infty$ such that 
$V_A(x)=\tV_A(x)=U(x)$ for $x \in D_A$
and such that $\sup_{x \in M_{R_A}} |V_A(x)-\tV_A(x)| \leq \e$, where $\e=\min\{\alpha-c_*,1\}/4$. 
We thus have 
$\min_{M_{R_A}} V_A =0$ and $\{x \in M_{R_A} : V_A(x)=0\}=\{x \in \R^d : U(x)=0\}$, 
because $\min U=0$, because $U=V_A$ on $D_A$ and because 
$U\geq m_A>0$ and $V_A \geq \tV_A-\e=m_A-\e\geq 1 - 1/4>0$ outside $D_A$.
Finally, we also 
have $c_*^{V_A} \leq c_*^{\tV_A}+3\e\leq c_*+3\e<\alpha$, since $\e \leq (\alpha-c_*)/4$. 
This ends the step.

\vip

{\it Step 2.} For each $A > \max\{1,|x_0|\}$, we consider
the inhomogeneous $M_{R_A}$-valued diffusion
\begin{align}\label{edsmod}
Y^A_t = x_0+ B_t - \frac12 \intot \beta_s \nabla V_A(Y^A_s)\dd s \quad \hbox{ modulo } 2R_A,
\end{align}
where, for $x=(x_1,\dots,x_d) \in \R^d$, 
$$
x \hbox{ modulo } 2R_A=
\Big(x_i - 2R_A\Big\lfloor \frac{x_i+R_A}{2R_A}\Big\rfloor \Big)_{i=1,\dots,d} 
\in [-R_A,R_A)^d.
$$
This is a $M_{R_A}$-valued time-inhomogeneous diffusion, starting from $x_0 \in M_{R_A}$, 
with time-dependent generator
$\cL_{t}\phi(y)=\frac12\ddiv[\nabla \phi](y)-\frac12\beta_t \nabla V_A(y)\cdot\nabla \phi(y)$. 
By Theorem \ref{HKS} and since $c>c_*^{V_A}$ by Step 1, 

\vip

(a) $V_A(Y^A_t) \to 0$ in probability as $t\to \infty$;

\vip

(b) if $\alpha \in (c_*^{V_A},c)$ and if $G$ is an open connected subset of $M_{R_A}$ such that 
$\{x\in M_{R_A} : V_A(x)=0\}\subset G$, 
$\partial G \subset \{x\in M_{R_A} : V_A(x)\geq \alpha\}$ and $x_0 \in G$,
then $\PR(\forall \; t\geq 0 , \; Y_t^A \in G)>0$.

\vip

{\it Step 3.} For each $A>\max\{1,|x_0|\}$ set 
$\Omega_A=\{\zeta=\infty,\;\sup_{t\geq 0} |X_t| < A\}$.
It holds that $\Omega_A=\{\sup_{t\geq 0} |Y_t^A| < A\}$ and
$\Omega_A \subset \{ \forall \; t\geq 0, \;\; X_t=Y^A_t\}$. 

\vip

Indeed, let $\tau_A=\inf\{t\geq 0 : |X_t|\lor |Y^A_t| > A\}$.
Since $R_A>A$, the modulo $2R_A$ is not active in \eqref{edsmod} during $[0,\tau_A]$.
Then a simple computation, using that $V_A=U$ on $B(A)$ and that
$\nabla U$ is lipschitz continuous on $B(A)$, with
Lipschitz constant $C_A$, shows that a.s., for all $t\geq 0$,
$$|X_{t\land \tau_A}-Y^A_{t\land \tau_A}| \leq C_A \intot \beta_s |X_{s\land \tau_A}-Y^A_{s\land \tau_A}| \dd s.$$ 
Since $(\beta_t)_{t\geq 0}$ is locally bounded,
$\sup_{t\geq 0}  |X_{t\land \tau_A}-Y^A_{t\land \tau_A}|=0$ a.s. by the Gronwall lemma.
Hence $X$ and $Y^A$ coincide until one of them (and thus both of them)
reaches $A$. The conclusion follows.

\vip

{\it Proof of (ii).} We fix $\e>0$. For $A> \max\{1,|x_0|\}$, by
Step 3 and since $V_A=U$ on $B(A)$,
$$
\PR(U(X_t) \geq \e) \leq \PR(\Omega_A^c) + \PR(U(X_t)\geq \e, \Omega_A) 
\leq \PR(\Omega_A^c) + \PR(V_A(Y^A_t)\geq \e) .
$$ 
By Step 2-(a), we conclude that 
$\limsup_{t\to \infty} \PR(U(X_t) \geq \e) \leq \PR(\Omega_A^c)$
for each $A\geq \max\{1,|x_0|\}$.
But by assumption, $\PR(\Omega_A^c)\to 0$ as $A\to \infty$, whence the conclusion. 

\vip

{\it Proof of (i).} We fix $G$ as in the statement. Consider 
$A> \max\{1,|x_0|\}$ such that $G \subset B(A)$. 
We thus have $G\subset M_{R_A}$, $\{V_A=0\}= \{U=0\}\subset G$,
and $\partial G \subset \{U\geq \alpha\}\cap B(A)=\{V_A\geq \alpha\}\cap B(A)$.
We then know by Step 2-(b)
that the event $\Omega_A'=\{\forall \; t\geq 0,\; Y^A_t \in G\}$ has a positive probability.
Using now that $\Omega_A' \subset \{\sup_{t\geq 0} |Y_t^A| < A\}=\Omega_A \subset 
\{ \forall \; t\geq 0, \;\; X_t=Y^A_t\}$ by Step 3, we deduce that we also have 
$\Omega_A' \subset\{\zeta=\infty$ and $\forall \; t\geq 0,\; X_t \in G\}$.
Thus $\PR(\zeta=\infty$ and $\forall \; t\geq 0,\; X_t \in G)>0$  as desired.
\end{proof}

We next recall some well-known facts concerning Bessel processes.

\begin{prop}\label{rap} Fix $\delta>0$, $r>0$ and 
let $(W_t)_{t\geq 0}$ be a $1$-dimensional Brownian motion.
Consider the pathwise unique solution $(R_t)_{t\geq 0}$, killed when 
it reaches $0$, to 
$$
R_t = r+ W_t + \frac{\delta-1}{2}\intot \frac{\dd s}{R_s}.
$$
Such a process is called a (killed) Bessel process with dimension $\delta$
starting from $r$.

\vip

(a) If $\delta \in (0,2)$, $(R_t)_{t\geq 0}$ a.s. reaches $0$.

\vip

(b) If $\delta\geq 2$, $(R_t)_{t\geq 0}$ does a.s. never reach $0$.

\vip

(c) If $\delta\geq 2$, we a.s. have
$\limsup_{t\to \infty} (t \log t)^{-1/2} R_t = 0$ a.s.

\vip

(d) If $\delta>2$, we a.s. have $\liminf_{t\to \infty} t^{-1/2}(\log t)^\nu R_t =\infty$,
where $\nu=4/(\delta-2)$.
\end{prop}

We refer to Revuz-Yor \cite[Chapter XI]{ry} for (a) and (b).
For (c), we actually have the more precise estimate $\limsup_{t\to \infty} 
(2t \log \log t)^{-1/2} R_t = 1$ a.s., see \cite[Chapter XI, Exercise 1.20]{ry}. 
Finally, (d) is proved in Motoo \cite{m}, when $\delta\geq 3$ is an integer, 
as a corollary of a
general result about diffusion processes 
that also applies to the case where $\delta>2$ is not an integer. More precisely, 
we have
$\liminf_{t\to \infty} t^{-1/2}f(t) R_t =\infty$ a.s. if $f:\R_+\to [1,\infty)$ is increasing and satisfies
$\int_0^\infty (1+t)^{-1}[f(t)]^{(2-\delta)/2} \dd t<\infty$.
See also Pardo-Rivero \cite[Subsection 2.3]{pr}, where this result is stated 
in terms of {\it squared} Bessel processes.

\section{Proof under $(H_1)$}\label{hh1}

First, we verify that the solution to \eqref{eds} is global and that it always
comes back in $B(A_0)$, where $A_0\geq 2$ was introduced in $(H_1(a))$.
This lemma really uses that $a$ is large enough.

\begin{lem}\label{lemh11}
Assume $(A)$, fix $c>0$ and $\beta_0\geq 0$, and suppose $(H_1(a))$ for some $a>0$. 
Consider the unique
maximal solution $(X_t)_{t\in [0,\zeta)}$ to \eqref{eds}, see Remark \ref{nn}.

\vip

(i) The solution is global, i.e. $\PR(\zeta=\infty)=1$.

\vip

(ii) If $a>c(d-2)/2$, for all $r\geq 0$, all $x\in \R^d \setminus B(A_0)$, 
$\PR(\inf\{t\geq r : |X_{t}|=A_0\}<\infty | X_{r}=x)=1$.
\end{lem}

\begin{proof}
By $(A)$ and $(H_1(a))$, there is $C>0$ such that $x\cdot\nabla U(x) \geq -C$ for all $x\in \R^d$.
For $n\in \N$, we define 
$\tau_n=\inf\{t>0 : |X_t|\geq n\}$. By It\^o's formula, we have, for any $T>0$,
$$\E[|X_{T\land \tau_n}|^2] \leq |x_0|^2+ dT+C\int_0^T \beta_s \dd s=:C_T.$$ 
Consequently, $\PR(\tau_n\leq T)\leq 
\PR(|X_{T\land \tau_n}|\geq n)\leq C_T/n^2$, so that $\zeta= \lim_n \tau_n=\infty$ a.s., which proves (i).
Concerning (ii), we fix $|x|>A_0\geq 2$ and $r\geq 0$ and we split the proof into several parts.

\vip

{\it Step 1.} Conditionally on $X_r=x$, the process 
$(\tX_t)_{t\geq 0}:=(X_{t+r})_{t\geq 0}$ solves \eqref{eds} with
$x_0$, $(\beta_{t})_{t\geq 0}$ 
and 
$(B_t)_{t\geq 0}$ replaced by $x$, $(\beta_{r+t})_{t\geq 0}$ and 
$(\tB_t)_{t\geq 0}:=(B_{t+r}-B_r)_{t\geq 0}$.
Observe that $(\tX_t)_{t\geq 0}$ 
does never hit $0$ by the Girsanov theorem
and since $d\geq 2$. We thus may use It\^o's formula to compute
$$
|\tX_{t}|= |x| + W_t + \int_0^{t} \Big(\frac{d-1}{2 |\tX_s|} - 
\frac{\beta_{r+s}\tX_s\cdot \nabla U(\tX_s)}{2|\tX_s|}
\Big) \dd s,
$$
where $W_t:=\int_0^t \frac{\tX_s\cdot \dd \tB_s}{|\tX_s|}$ is a 
$1$-dimensional Brownian motion.
We define $\rho=\inf\{t\geq 0 : |\tX_t|=A_0\}$ and recall that our goal is to prove that
$\rho<\infty$ a.s. We next introduce $(S_t)_{t\in [0,\sigma]}$ solving
$$
S_t = |x| + W_t + \int_0^{t} \Big(\frac{d-1}{2 S_s} - \frac{a \beta_{s}}{2S_s \log S_s}
\Big) \dd s\quad \hbox{killed at} \quad \sigma=\inf\{t\geq 0 : S_t=A_0\}.
$$
We claim that $\{\rho = \infty\} \subset \{\sigma=\infty\}$. Indeed, 
using $(H_1(a))$ and that $\beta_{r+s}\geq \beta_s$, one checks that
$[\beta_{r+s}\tX_s\cdot \nabla U(\tX_s)]/[2|\tX_s|]\geq 
[a \beta_{s}]/[2|\tX_s|\log|\tX_s|]$ for all $s\in [0,\rho]$. 
Hence, setting $b(s,r)=(d-1)/(2r) - [a \beta_s]/[2r\log r]$ for $s\geq 0$ and $r\geq A_0$,
we have
$$
\frac{\dd}{\dd t} (|\tX_t|-S_t) \leq (b(t,|\tX_t|)-b(t,S_t))
$$
for all $t\in [0,\rho\land\sigma)$. Hence, setting $z_+=\max\{z,0\}$,
\[\frac{\dd}{\dd t} (|\tX_t|-S_t)_+^2 \leq 2(|\tX_t|-S_t)_+ 
(b(t,|\tX_t|)-b(t,S_t)) 
\leq  2 C_t (|\tX_t|-S_t)_+ \big| |\tX_t|-S_t \big|  
= 2 C_t (|\tX_t|-S_t)_+^2,\]
$C_t$ being the global Lipschitz constant of $r \mapsto b(t,r)$ on $[A_0,\infty)$. Since 
$S_0=|\tX_0|$ and since 
$t\mapsto C_t$ is locally bounded on $[0,\infty)$, we conclude that 
$(|\tX_t|-S_t)_+^2=0$ for all $t\in
[0,\rho\land \sigma)$ a.s.
Hence, on the event $\{\rho = \infty\}\cap \{\sigma<\infty\}$, 
we have $S_t \geq |\tX_t|>A_0$ for all $t\in [0,\sigma]$, whence $\sigma=\infty$.
Thus $\{\rho = \infty\}\subset \{\sigma=\infty\}$, so that
our goal is from now on to verify that $\sigma<\infty$ a.s.

\vip

{\it Step 2.} We next introduce the Bessel process $(R_t)_{t \geq 0}$
$$
R_t= |x|+ W_t + \frac{d-1}{2}\int_0^t \frac{\dd s}{R_s}.
$$
Since $d\geq 2$, we know from Proposition \ref{rap}-(b) that $R_t$ does never reach $0$.
It holds that a.s., $S_t \leq R_t$ for all 
$t\in [0,\sigma)$\footnote{
If $d=2$, one may conclude here, since $R$ is then recurrent.}:
it is sufficient to use that $S_0=R_0$ and 
that for $t\in [0,\sigma)$, since $a\beta_t/[2S_t\log S_t]\geq 0$,
$$
\frac \dd{\dd t} (S_t-R_t)_+^2 \leq (d-1)(S_t-R_t)_+
\Big(\frac 1{S_t}-\frac 1{R_t}\Big)\leq 0.
$$
By Proposition \ref{rap}-(c), 
$\limsup_{t\to \infty} (t \log t)^{-1/2}R_t =0$ a.s., so that
\[\liminf_{t\to \infty} \log t/\log R_t  \geq 2,\]
 whence 
$$\{\sigma = \infty\}\subset 
\{\liminf_{t\to \infty} \log t/\log S_t  \geq 2\}\subset
\{\liminf_{t\to \infty} \beta_t/\log S_t  \geq 2/c\}.
$$
We fix $\eta\in(0,1)$ such that $\delta:=d-2a(1-\eta)/c \in (0,2)$, 
which is possible because
$a>c(d-2)/2$. We then know that 
$\tau= \inf\{t>0 : \forall\; s\geq t, \; \beta_s/\log S_s  \geq 2(1-\eta)/c \}$ 
is a.s. finite on $\{\sigma=\infty\}$.

\vip

{\it Step 3.} We now fix $K\geq 1$ and $L> A_0$ and we introduce 
$\Omega_{K,L}=\{\sigma= \infty,\tau \leq K, S_K \leq L\}$, as well as the
Bessel process $(S^{K,L}_t)_{t\geq L}$
with dimension $\delta \in (0,2)$, issued from $L$ at time $K$:
for all $t\geq K$,
$$
S^{K,L}_t= L + (W_t-W_K) + \frac{\delta-1}{2} \int_K^t \frac{\dd s}{S^{K,L}_s}\quad
\hbox{killed at $\sigma_{K,L}=\inf\{t\geq K : S^{K,L}_t=A_0\}$.} 
$$
We claim that $\Omega_{K,L} \subset \{\sigma_{K,L}=\infty\}$.
By definition of $\tau$ and since $\delta=d-2a(1-\eta)/c$, we see 
that on $\Omega_{K,L}$, for all $t\in [K,\sigma_{K,L})$, we have
$$
\frac \dd{\dd t}(S_t-S^{K,L}_t)=\frac{d-1}{2 S_t} - 
\frac{a \beta_t}{2S_t\log S_t} - \frac{\delta-1}{2S^{K,L}_t}
\leq \frac{\delta-1}2\Big(\frac1{S_t}-\frac1{S^{K,L}_t} \Big),
$$
whence, for all $t\geq K$,
$$
\frac \dd{\dd t}(S_t-S^{K,L}_t)_+^2\leq (\delta-1)(S_t-S^{K,L}_t)_+ 
\Big(\frac1{S_t}-\frac1{S^{K,L}_t} \Big)
\leq 0.
$$
Since furthermore $S_K\leq L$ on $\Omega_{K,L}$, we have $(S_K-S^{K,L}_K)_+=0$, so that,
still on $\Omega_{K,L}$,
$S^{K,L}_t\geq S_t>A_0$ for all $t\in [K,\sigma_{K,L})$, whence $\sigma_{K,L}=\infty$
(else, we would have $A_0=S^{K,L}_{\sigma_{K,L}}\geq S_{\sigma_{K,L}}>A_0$).

\vip

{\it Step 4.} But we know from Proposition \ref{rap}-(a), since $\delta \in (0,2)$, that
$\sigma_{K,L}<\infty$ a.s. We conclude that for all $K\geq 1$, all $L>A_0$, 
$\PR(\Omega_{K,L})=0$, i.e. $\PR(\sigma= \infty,\tau \leq K, S_K \leq L)=0$.
Letting $L\to \infty$, we find that  $\PR(\sigma= \infty,\tau \leq K)=0$,
since $S_K<\infty$ a.s. on $\{\sigma=\infty\}$. Letting $K\to \infty$,
we deduce that $\PR(\sigma= \infty)=0$, since $\tau<\infty$ a.s. on $\{\sigma=\infty\}$ 
by Step 2.
The proof is complete.
\end{proof}

We now bound from below the probability to remain stuck forever in a certain ball
when starting from the circle with radius $A_0$.

\begin{lem}\label{lemh12}
Assume $(A)$, fix $c>0$ and $\beta_0\geq 0$, and suppose $(H_1(a))$ for some $a>0$. 
Consider the unique
(global by Lemma \ref{lemh11}) solution $(X_t)_{t\geq 0}$ to \eqref{eds}.
There is $B>A_0$ such that 
$$
p:=\inf_{r \geq 0, |x|=A_0}  \PR\Big(\sup_{t\geq r} |X_t| < B \Big|X_{r}=x\Big) >0.
$$
\end{lem}

\begin{proof} In view of Lemma \ref{lemzero}-(i), the only difficulty 
is get the uniformity in $r\geq 0$
and $|x|=A_0\geq 2$.
\vip

{\it Step 1.} We fix $r\geq 0$ and $x\in \R^d$ such that $|x|=A_0$ and 
we set $(\tX_t)_{t\geq 0}=(X_{r+t})_{t\geq 0}$.
Exactly as in the first step of the previous proof, we can write, conditionally
on $X_r=x$,
$$
|\tX_t|=A_0+ W_t+\int_0^t \Big(\frac{d-1}{2|\tX_s|} - 
\frac{\beta_{r+s} \tX_s\cdot \nabla U(\tX_s)}{2|\tX_s|}\Big) \dd s.
$$
We claim that a.s., $|\tX_t|\leq A_0+R_t$ for all $t\geq 0$, 
where $R_t$ is $(0,\infty)$-valued and solves
$$
R_t= 1+ W_t + \int_0^t \Big(\frac{d-1}{2R_s} - \beta_{s} b(R_s)\Big) \dd s,
$$
with $b(r)= a r /[4(A_0^2+r^2) \log(A_0^2+r^2)]$. The fact that
$R$ does never reach $0$ follows from Proposition \ref{rap}-(b), since $d\geq 2$,
and from the Girsanov theorem, since $b$ is bounded.

\vip

To check this claim, we first
observe that, thanks to $(H_1(a))$,
$$
|x|\geq A_0 \;  \Longrightarrow\;
\frac{\beta_{r+t} x \cdot \nabla U(x)}{2|x|} \geq 
\frac{ a \beta_{s}}{2|x| \log |x|} \geq \beta_s b(|x|-A_0),
$$
the last inequality following from the fact that $b(r)\leq a/[2(A_0+r)\log(A_0+r)]$
for all $r\geq 0$, 
because
$\log(A_0^2+r^2)\geq \log (A_0+1+r^2) \geq \log(A_0+r)$ and 
$4(A_0^2+r^2) / r \geq 2 (A_0+r)^2/(A_0+r)= 2 (A_0+r)$.
Consequently, using that $(|\tX_t|-A_0-R_t)_+>0$ implies that 
$|\tX_t|\geq A_0+R_t\geq A_0$, we see that
\begin{align*}
\frac \dd{\dd t} (|\tX_t|-A_0-R_t)_+^2 =&\ 2(|\tX_t|-A_0-R_t)_+ 
\Big[ \frac{d-1}{2|\tX_t|}   - \frac{\beta_{r+t} \tX_t\cdot \nabla U(\tX_t)}{2|\tX_t|}
    -\frac{d-1}{2R_t} +\beta_t b(R_t)\Big] \\
\leq&\ 2(|\tX_t|-A_0-R_t)_+ 
\Big[ \frac{d-1}{2}\Big(\frac{1}{|\tX_t|}
-\frac{1}{R_t} \Big) - \beta_t [b(|\tX_t|-A_0) - b(R_t)]\Big]\\
\leq &\  - 2 \beta_t(|\tX_t|-A_0-R_t)_+[b(|\tX_t|-A_0) - b(R_t)] \\
\leq &\  2 C \beta_t (|\tX_t|-A_0-R_t)_+ \big||\tX_t|-A_0-R_t\big|\\
\leq &\  2C \beta_t (|\tX_t|-A_0-R_t)_+^2,
\end{align*}
where $C$ is the (global) Lipschitz constant of $b$. 
The claim follows, since $|\tX_0|-A_0-R_0=-1\leq0$.

\vip

{\it Step 2.} Since the law of $(R_t)_{t\geq 0}$ does not depend on $x$ such that
$|x|=A_0$ nor on $r\geq 0$, it suffices to check that there is $K>0$ such that
$\PR(\sup_{t\geq 0} R_t \leq K)>0$. By Step 1, the conclusion, with $B=A_0+K$, will follow.

\vip

Set $V(y)=a \log \log (A_0^2 + |y|^2)/ 4 -a \log \log (A_0^2)/ 4 $ 
for all $y\in \R^d$, consider $y_0\in\R^d$ such that $|y_0|=1$, as well as the diffusion
process
$$
Y_t=y_0+ B_t - \frac 12 \intot \beta_s \nabla V(Y_s)\dd s.
$$
Observe that $V$ satisfies $(A)$ with 
$c_*=0$. We consider now the bounded connected open set $G=B(K)$, where $K>1$ 
is large enough so that
for $y \in \partial G$, $V(y)=a \log \log (A_0^2+K^2)/4 - a \log \log (A_0^2)/4 >c$. 
We also have $\{y \in \R^d : V(y)=0\}=\{0\} \subset G$.
By Lemma \ref{lemzero}-(i), since $y_0 \in G$, we conclude that 
$\PR(\forall\;t\geq 0,\;|Y_t|<K )>0$.
Finally, one can check that $(|Y_t|)_{t\geq 0}=(R_t)_{t\geq 0}$
in law, by applying the It\^o formula,
using that $\frac{y}{2|y|} \cdot \nabla V(y)=b(|y|)$. 
All in all, $\PR(\sup_{t\geq 0} R_t \leq K)>0$ as desired.
\end{proof}

We can now give the

\begin{proof}[Proof of Theorem \ref{mr} under $(A)$ and $(H_1(a))$ with $a>c(d-2)/2$.]
We consider the solution $(X_t)_{t\geq 0}$ to \eqref{eds}, 
which is global by Lemma \ref{lemh11}-(i),
denote by $\cF_t=\sigma(\{X_s, s\in [0,t]\})$, and recall that $B>A_0$ and $p>0$ were 
introduced in Lemma \ref{lemh12}.
We introduce the sequence of stopping times 
$S_0\leq T_1\leq S_1\leq T_2 \leq S_2\leq ...$, with
$S_0=0$ and, for all $n\geq 0$, $T_{n+1}=\inf\{t>S_n : |X_t|\geq B\}$ 
and $S_{n+1}=\inf\{t>T_{n+1} : |X_t|\leq A_0\}$,
with the convention that $\inf \emptyset = \infty$. 
In particular, $T_n=\infty$ implies that
$S_k=T_k=\infty$ for all $k\geq n$. Our goal is to verify that 
a.s., there is $N\geq 1$ such that
$T_N = \infty$, implying that $\limsup_{t\to\infty} |X_t|\leq B$, so that
$\sup_{t \geq 0} |X_t| <\infty$ a.s.,
whence the conclusion by Lemma \ref{lemzero}-(ii).

\vip

Using the strong Markov property, one deduces that for all $n\geq 1$,
$\{T_n <\infty\} \subset \{S_n<\infty\}$ by Lemma \ref{lemh11}-(ii), while 
$\PR(T_{n+1}=\infty | \cF_{S_n}  ) \geq p$ on the event $\{S_n<\infty\}$ 
by Lemma \ref{lemh12}.
Hence for all $n\geq 1$,
$$
\PR(T_{n+1}<\infty) = 
\E[\indiq_{\{S_n<\infty\}} \PR(T_{n+1}<\infty | \cF_{S_n}  )] \leq (1-p)\PR(S_{n}<\infty)=
(1-p)\PR(T_{n}<\infty).
$$
Hence $\PR(\cap_{k\geq 1}\{T_k<\infty\})=\lim_{n\to\infty}\PR(\cap_{k=1}^n\{T_k<\infty\})=
\lim_{n \to \infty} \PR(T_n <\infty)=0$ as desired.
\end{proof}

\section{Proof under $(H_2)$}\label{hh2}

Under $(H_2)$, the proof is rather simpler. It entirely relies on
the following lemma.

\begin{lem}\label{troc}
Consider a $1$-dimensional Brownian motion
$(W_t)_{t\geq 0}$. For $c>0$ and $\delta>0$, consider the $(0,\infty)$-valued 
pathwise unique solution $(S^{\delta}_t)_{t\geq 0}$ to
$$
S^{\delta}_t =c \delta +  W_t + \frac{d-1}{2} \int_0^t \frac{\dd s}{S_s^{\delta}}
- \frac{1}{2c \delta }\intot \log(1+s) \dd s.
$$
For any $\delta_0>0$ and $\eta>0$, it holds that
$$
p(\delta_0,\eta)=\inf_{\delta \geq \delta_0} 
\PR\Big(\sup_{t\geq 0} S_t^{\delta} \leq c \delta(1+\eta) \Big) >0.
$$
\end{lem}

\begin{proof}
The strict positivity of $S^\delta$
follows from Proposition \ref{rap}-(ii) and the Girsanov theorem,
since $d\geq 2$ and since the additional drift is bounded (locally in time).
First, $R^\delta_t=(\delta_0/\delta)S^\delta_{(\delta/\delta_0)^2 t}$ solves
\begin{equation}\label{troce}
R^\delta_t= c\delta_0 + W^{\delta}_t+ \frac{d-1}{2}\intot\frac{\dd s}{R_s^{\delta}}
-\frac1{2c\delta_0}\intot\log(1+(\delta/\delta_0)^2s) \dd s,
\end{equation}
where $W^{\delta}_t=(\delta_0/\delta)W_{(\delta/\delta_0)^2 t}$ is a Brownian motion. 
We introduce $H^\delta$ solving \eqref{troce} with $W^\delta$ replaced by $W$. We claim that
for any $\delta\geq \delta_0$, $H^\delta_t \leq M_t$ for all $t\geq 0$ a.s., where $M$ solves
$$
M_t=c\delta_0 + W_t+ \frac{d-1}{2}\intot\frac{\dd s}{M_s}
-\frac1{2c\delta_0}\intot\log(1+ s)b(M_s) \dd s,
$$
where $b(r)=(\e^2+r^2)^{-1/2}r$, with $\e=c\delta_0[(1+\eta)^2-(1+\eta/2)^2]/(2+\eta)$. 
Indeed, we write as usual
$$
\frac \dd{\dd t} (H^{\delta}_t-M_t)_+^2 =(H^{\delta}_t-M_t)_+
\Big((d-1)\Big[\frac 1{H^{\delta}_t}-\frac 1{M_t}\Big] + \frac1{c\delta_0}
\Big[\log(1+t)b(M_t)-\log(1+(\delta/\delta_0)^2t)
\Big]\Big)\leq 0.
$$
Hence for all $\delta\geq \delta_0$, we have
\begin{align*}
\PR\Big(\sup_{t\geq 0} S_t^{\delta} <c \delta(1+\eta) \Big) \ &=\  
\PR\Big(\sup_{t\geq 0} R_t^{\delta} < c \delta_0 (1+\eta) \Big) \\ 
&=\ \PR\Big(\sup_{t\geq 0} H_t^{\delta} < c \delta_0(1+\eta) \Big) \\
&\geq\ \PR\Big(\sup_{t\geq 0} M_t < c\delta_0 (1+\eta) \Big)
\end{align*}
and it suffices to prove that $p:=\PR(\sup_{t\geq 0} M_t < c\delta_0 (1+\eta) )>0$.

\vip

We introduce $V(y)=(\e^2+|y|^2)^{1/2}-\e$, which satisfies $(A)$ with $c_*=0$.
We consider $y_0\in \R^d$ such that $|y_0|= c\delta_0$, as well as the diffusion process
$$ 
Y_t=y_0+ B_t - \frac 12 \intot \frac{\log (1+ s)}{c \delta_0} \nabla V(Y_s)\dd s.
$$
One can check that $(|Y_t|)_{t\geq 0}=(M_t)_{t\geq 0}$ in law, using the It\^o formula and that
$\frac{y}{|y|} \cdot \nabla V(y)=b(|y|)$. We then observe that
$V$ satisfies $(A)$ with $c_*=0$ and we consider the the bounded 
connected open set $G=B(c\delta_0(1+\eta))$.
We have $y_0 \in G$, $\{y \in \R^d : V(y)=0\}=\{0\} \subset G$ and, by definition of $\e$,
$\partial G \subset \{y \in \R^d : V(y)=c\delta_0 (1+\eta/2)\}$.
Applying Lemma \ref{lemzero}-(i) (with $c$ replaced by $c\delta_0>0=c_*$),
we conclude that 
$\PR(\forall\;t\geq 0,\;|Y_t|<c\delta_0 (1+\eta) )
=\PR(\forall\;t\geq 0,\;Y_t\in G)>0$ as desired.
\end{proof}

Once this is seen, we can give the

\begin{proof}[Proof of Theorem \ref{mr} under $(A)$ and $(H_2(\alpha))$ with $\alpha>c$.]
We consider the solution $(X_t)_{t\in [0,\zeta)}$ to \eqref{eds} as in Remark \ref{nn}.
By Lemma \ref{lemzero}-(ii), we only have to verify that a.s., $\zeta=\infty$
and $\sup_{t\geq 0} |X_t|<\infty$.

\vip

{\it Step 1.} In Assumption $(H_2(\alpha))$, we have $\lim_{i\to \infty} a_i=\infty$,
because $a_{i+1}\geq b_i\geq a_i+\alpha\delta_i\geq a_i+\alpha\delta_0$.
We thus may consider $i_0\geq 1$ such that $a_{i_0}>|x_0|$.
We introduce $\cF_t=\sigma(X_{s}\indiq_{\{\zeta>s\}}, s\in [0,t])$, as well as 
the sequence of stopping times
$T_i=\inf\{t \geq 0 : |X_t|=a_i\}$ and $S_i=\inf\{t \geq 0 : |X_t|=b_i\}$, for $i\geq i_0$,
with the usual convention that $\inf \emptyset =\infty$. We have 
$0<T_{i_0}\leq S_{i_0} \leq T_{i_0+1}\leq S_{i_0+1} \dots$.
It suffices to prove that $\lim_{i \to \infty} \PR(S_i < \infty)=0$. 

\vip

Indeed, this will imply that
$\PR(\cap_{k\geq 1}\{S_k<\infty\})=
\lim_{i \to \infty} \PR(S_i <\infty)=0$, so that there will a.s. exist $I$  such that 
$S_I=\infty$.
Consequently, we will a.s. have $\zeta=\infty$ and $\sup_{t\geq 0} |X_t| \leq b_I <\infty$.

\vip

{\it Step 2.}
We fix $\eta>0$ such that $\alpha>c(1+\eta)$. It suffices to verify that 
for all $i\geq i_0$, we have $\PR(S_i=\infty | \cF_{T_i})\geq p(\delta_0,\eta)$
on $\{T_i<\infty\}$, where $p(\delta_0,\eta)$ was defined in Lemma \ref{troc}.

\vip

Indeed, this will imply that $\lim_{i \to \infty} \PR(S_i < \infty)=0$, because 
for all $i\geq i_0+1$,
\begin{align*}
\PR(S_i<\infty)=\E[\indiq_{\{T_i<\infty\}} \PR(S_i<\infty | \cF_{T_i})]
\ &\leq\ (1-p(\delta_0,\eta)) \PR(T_{i}<\infty)\\
& \leq \ (1- p(\delta_0,\eta)) \PR(S_{i-1}<\infty).
\end{align*}

{\it Step 3.} To conclude, we fix $i\geq i_0$ and apply the It\^o formula:
on $\{T_i<\infty\}$, for $t\in [0,\zeta-T_i)$,
$$
|X_{T_i+t}| = a_i + W^i_{t} + \frac{d-1}2 \int_{0}^{t} \frac{\dd s}{|X_{T_i+s}|}
- \frac 12\intot \frac{\beta_{T_i+s} X_{T_i+s}\cdot \nabla U(X_{T_i+s}) }{|X_{T_i+s}|}\dd s,
$$
the Brownian motion $W^i_t=\int_{T_i}^{T_i+t} (\frac{X_s}{|X_s|}\indiq_{\{s<\zeta\}}+ 
\bu\indiq_{\{s\geq \zeta\}})
\cdot \dd B_s$ being independent of $\cF_{T_i}$. Here we introduced some 
arbitrary deterministic unitary vector
$\bu \in \R^d$. We introduce, still on $\{T_i<\infty\}$,
$$
R^i_t= c\delta_i+W^i_t + \frac{d-1}2 \int_{0}^{t} \frac{\dd s}{R^i_s}
- \frac 1{2c\delta_i } \intot \log(1+s)\dd s.
$$
This process is well-defined, see Lemma \ref{troc}.
We now check that a.s. on  $\{T_i<\infty\}$, it holds that $|X_{T_i+t}| \leq a_i+R^i_t$ 
for all $t \in [0,S_i-T_i)$. Observe that this makes sense, because $\zeta>S_i$.

\vip

Using that $t\in[0,S_i-T_i)$ implies that $|X_{T_i+t}|< b_i$, that 
$(|X_{T_i+t} |-a_i-R^i_t)_+>0$ implies that $|X_{T_i+t}|>a_i$, and that 
$|x|\in[a_i,b_i]$ implies that $x\cdot\nabla U(x)\geq |x|/\delta_i$,
we see that for all $t \in [0,S_i-T_i)$
\begin{align*}
&\frac \dd{\dd t} (|X_{T_i+t} |-a_i-R^i_t)_+^2\\
=&\  (|X_{T_i+t} |-a_i-R^i_t)_+\Big((d-1)\Big[\frac1{|X_{T_i+t}|}
-\frac 1{R^i_t}\Big] - \frac{\beta_{T_i+t} X_{T_i+t}\cdot \nabla U(X_{T_i+t}) }{|X_{T_i+t}|} + 
\frac {\log (1+t)}{c \delta_i }\Big)\\
\leq&\  (|X_{T_i+t} |-a_i-R^i_t)_+\Big(-\frac{\beta_{T_i+t}}{\delta_i} +  
\frac {\log (1+t)}{c \delta_i }\Big) \leq 0,
\end{align*}
since finally $\beta_{T_i+t}\geq c^{-1}\log(1+t)$.
Since $|X_{T_i} |-a_i-R^i_0=-c\delta_i\leq 0$, we conclude that indeed, 
on  $\{T_i<\infty\}$, 
it holds that $|X_{T_i+t}| \leq a_i+R^i_t$ for all $t \in [0,S_i-T_i)$.

\vip

Hence 
\[
\{T_i<\infty,\sup_{t\geq 0}R^i_t < c\delta_i(1+\eta)\}
\subset {\{T_i<\infty, \sup_{t\in [0,S_i-T_i)}|X^i_{T_i+t}| < a_i+c\delta_i(1+\eta)\}},
\] 
which is included in 
$\{S_i=\infty\}$ since finally $a_i+c\delta_i(1+\eta)<a_i+\alpha \delta_i = b_i$
by assumption. Hence on $\{T_i<\infty\}$,
$$
\PR(S_i=\infty|\cF_{T_i})\geq \PR\Big(\sup_{t\geq 0}R^i_t < c\delta_i(1+\eta)\Big) 
\geq p(\delta_0,\eta)
$$
by Lemma \ref{troc}, since $\delta_i\geq \delta_0$. The proof is complete.
\end{proof}

\section{Proof under $(H_3)$}\label{hh3}

The proof under $(H_3)$ is very similar, in its principle, to the proof under $(H_2)$.
We start with the following variation of Lemma \ref{troc}.

\begin{lem}\label{troc3}
Consider a $1$-dimensional Brownian motion
$(W_t)_{t\geq 0}$. For $c>0$ and $\kappa>0$, consider the $(0,\infty)$-valued 
pathwise unique solution $(R^{\kappa}_t)_{t\geq 0}$ to
$$
R^{\kappa}_t =c +  W_t + \frac{d-1}{2} \int_0^t \frac{\dd s}{R_s^\kappa}
- \frac{1 }{ 2c}\intot \log(1+s/(2\kappa)^2) \dd s.
$$
For any $\eta>0$, any $\e>0$, it holds that
$$
q(\eta,\e)=\inf_{\kappa>0} 
\PR\Big(\sup_{t\geq 0} R_t^{\kappa} \leq c \max\{\e \kappa,1\}(1+\eta) \Big) >0.
$$
\end{lem}

\begin{proof}
As in Lemma \ref{troc}, $R^\kappa$ does never reach zero by Proposition \ref{rap} 
and the Girsanov theorem.

\vip

We observe that $T^\kappa_t=(2\kappa)^{-1}R^\kappa_{(2\kappa)^2 t}$ solves, 
with the brownian motion
$W^\kappa_t= (2\kappa)^{-1}W^\kappa_{(2\kappa)^2 t}$,
$$
T^{\kappa}_t =\frac c {2\kappa} +  W_t^\kappa + \frac{d-1}{2} \int_0^t \frac{\dd s}{T_s^{\kappa}}
- \frac{2\kappa }{ 2c}\intot \log(1+s) \dd s.
$$
Hence $T^\kappa$ has the same law as $S^\delta$, see Lemma \ref{troc}, with $\delta=1/(2\kappa)$.
Thus for all $\kappa \in (0,1/\e]$, 
$\PR(\sup_{t\geq 0} R_t^{\kappa} \leq c(1+\eta) )=
\PR(\sup_{t\geq 0} T_t^{\kappa} \leq (c/2\kappa)(1+\eta))\geq p(\e/2,\eta)$.

\vip

If now $\kappa \in (1/\e,\infty)$, since $c/(2\kappa)<c \e/2$ and $2\kappa \geq 2/\e$, we see
that, in law, $T^\kappa_t \leq S^{\e/2}_t$ by a comparison argument.
As a consequence, $\PR(\sup_{t\geq 0} R_t^{\kappa} \leq c\e \kappa (1+\eta))=
\PR(\sup_{t\geq 0} T_t^{\kappa} \leq c\e (1+\eta)/2)
\geq \PR(\sup_{t\geq 0} S_t^{\e/2} \leq c\e (1+\eta)/2) \geq  p(\e/2,\eta)$.
The conclusion follows with $q(\eta,\e)=p(\e/2,\eta)>0$ which is positive, see Lemma \ref{troc}.
\end{proof}

\begin{proof}[Proof of Theorem \ref{mr} under $(A)$ and $(H_3(\alpha,\beta_0))$ 
with $\alpha>c$.]
We consider the solution $(X_t)_{t\in [0,\zeta)}$ to \eqref{eds} as in Remark \ref{nn}.
By Lemma \ref{lemzero}-(ii), we only have to verify that a.s., $\zeta=\infty$
and $\sup_{t\geq 0} |X_t|<\infty$.

\vip

{\it Step 1.}
Since $\cZ_i^-$ increases to $\R^d$ by assumption, we can find $i_0\geq 1$ such that 
$x_0$ belongs to the interior of $\cZ_{i_0}^-$.
We also introduce $\cF_t=\sigma(X_{s}\indiq_{\{\zeta>s\}}, s\in [0,t])$, as well as
the sequence of stopping times  $T_i=\inf\{t \geq 0 : X_t \in \partial \cZ_i^-\}$ and 
$S_i=\inf\{t \geq 0 : X_t \in \partial \cZ_i^+\}$, for $i\geq i_0$,  
with the usual convention that $\inf \emptyset =\infty$.
Since $\cZ_{i_0}^- \subset (\cZ_{i_0}^+)^c \subset \cZ_{i_0+1}^- \subset (\cZ_{i_0+1}^+)^c \dots$,
we have $0<T_{i_0}\leq S_{i_0} \leq T_{i_0+1}\leq S_{i_0+1} \dots$.
It suffices to verify that $\lim_{i \to \infty} \PR(S_i < \infty)=0$. 

\vip

Indeed, this will tell us that
$\PR(\cap_{k\geq 1}\{S_k<\infty\})=0$. There will thus a.s. exist $I$
such that $S_I=\infty$, so that $X_t \in \cZ_{I+1}^-$ for all $t\geq 0$,
whence the conclusion, since $\cZ_{I+1}^-$ is bounded.

\vip

{\it Step 2.}
We fix $\eta>0$ such that $\alpha>c(1+\eta)$. 
As in the proof under $(H_2(\alpha))$, it is enough to verify that
for all $i\geq 1$, we have $\PR(S_i=\infty | \cF_{T_i})\geq q(\e,\eta)$
on $\{T_i<\infty\}$, where $q(\e,\eta)$ is defined in Lemma \ref{troc3}
and where $\e>0$ is the constant introduced in Assumption $(H_3(\alpha,\beta_0))$.

\vip

{\it Step 3.}
Recall Assumption $(H_3(\alpha,\beta_0))$ and that for $i\geq 1$,
$\cZ_i$ is $C^\infty$-diffeomorphic to the annulus $\cC=\{x\in \R^d : |x|\in (1,2)\}$.
It is a tedious but classical exercise to prove that 
for each $i\geq i_0$, we can build a smooth function $V_i:\R^d \to [0,\infty)$
such that $V_i=U$ on $\cZ_i$, such that $V_i \leq u_i$ on $\cZ_i^-$, $V_i\geq v_i$ on
$\cZ_i^+$, $|\nabla V_i| \leq 2\kappa_i$ on $\R^d$, and $\nabla V_i(x) \neq 0$ for all
$x\in \R^d \setminus \{x_0\}$.
Observe that since $\nabla U$ does not vanish on $\bar \cZ_i$, it holds that 
$V_i(x)=U(x)\in(u_i,v_i)$
for all $x \in \cZ_i$ because else, $U$ would have a local extremum inside $\cZ_i$.

\vip

{\it Step 4.} In this whole step, we fix $i\geq i_0$ and work on $\{T_i<\infty\}$. 
For all $t\in [0,\zeta-T_i)$,
\[V_i(X_{T_i+t}) = u_i + \int_0^t |\nabla V_i(X_{T_i+s})| \dd W^i_{s} 
+ \frac{1}2 \int_{0}^{t} \Big[ \Delta V_i(X_{T_i+s}) -\beta_{T_i+s} \nabla U(X_{T_i+s})
\cdot \nabla V_i(X_{T_i+s})\Big] \dd s,
\]
the Brownian motion $W^i_t=\int_{T_i}^{T_i+t} (\frac{\nabla V_i(X_{s})}{|\nabla V_i(X_{s})|}
\indiq_{\{s<\zeta\}} + \bu \indiq_{\{s\geq \zeta\}})\cdot \dd B_s$ being independent of 
$\cF_{T_i}$. 
We introduced some deterministic unit vector $\bu \in \R^d$ and used that a.s., 
$X_{t}\neq x_0$ for all $t\in (0,\zeta)$ (by the Girsanov theorem, recall \eqref{eds}
and that $d\geq 2$) 
and thus $|\nabla V_i(X_{t})|>0$ for all $t> 0$.
We next introduce the time-change 
$\theta^i_t=\intot|\nabla V_i(X_{T_i+s})|^2 \dd s$, which is continuous and strictly 
increasing on $[0,\zeta-T_i)$, as well as its inverse 
$\tau^i_t : [0,\theta^i_{\zeta-T_i})\to \R_+$.
For all $t\in [0,\theta^i_{\zeta-T_i})$, we have
\[
V_i(X_{T_i+\tau^i_t}) = u_i + \bar W^i_t  
+ \frac{1}2 \int_{0}^{t} \Big[ \frac{\Delta V_i(X_{T_i+\tau^i_s})}
{|\nabla V_i(X_{T_i+\tau^i_s})|^2} -
\beta_{T_i+\tau^i_s} \frac{\nabla V_i(X_{T_i+\tau^i_s})\cdot \nabla U(X_{T_i+\tau^i_s})}
{|\nabla V_i(X_{T_i+\tau^i_s})|^2}
\Big] \dd s,
\]
for some Brownian motion $\bar W^i$ independent of $\cF_{T_i}$, which can be built as 
follows: for a Brownian motion $\hat W$ independent 
of everything else (this is useless if $\theta^i_{\zeta-T_i}=\infty$ a.s.), set
$$
\bar W^i_t =\int_0^{\tau^i_{t} \land (\zeta-T_i) }
|\nabla V_i(X_{T_i+s})| \dd W^i_{s} + \int_{t\land\theta^i_{\zeta-T_i} }^t \dd \hat{W}_s,
$$
with the convention that $\tau^i_t\land (\zeta-T_i)=\zeta-T_i$ for all
$t\geq \theta^i_{\zeta-T_i}$.
We next introduce
$$
Y^i_t= c+\bar W^i_t + \frac{d-1}2 \int_{0}^{t} \frac{\dd s}{Y^i_s}
- \frac 1{2c} \intot \log(1+s/(2\kappa_i)^2)\dd s.
$$
This process is well-defined and positive, see Lemma \ref{troc3}. 
We could replace the strong repulsion term $(d-1)/(2Y^i_s)$ by a (weaker)
reflection term, but this allows us to make Lemmas \ref{troc} and \ref{troc3} more similar.
We now check that a.s. on  $\{T_i<\infty\}$, it holds that $V_i(X_{T_i+\tau^i_t})
\leq u_i+Y^i_t$ for all $t \in [0,\theta^i_{S_i-T_i})$.
Using that $t\in [0,\theta^i_{S_i-T_i})$ implies that $\tau^i_t < S_i-T_i$
and thus that $V_i(X_{T_i+\tau^i_t})< v_i$, that
$(V_i(X_{T_i+\tau^i_t})-u_i-Y^i_t)_+>0$ implies that $V_i(X_{T_i+\tau^i_t})>u_i$, and that 
$V_i(x) \in (u_i,v_i)$ implies that $x \in \cZ_i$, whence 
$\Delta V_i(x)/|\nabla V_i(x)|^2=\Delta U(x)/|\nabla U(x)|^2\leq \beta_0$ and
$\nabla V_i(x)\cdot \nabla U(x)/|\nabla V_i(x)|^2=1$,
we see that for all $t \in [0,\theta^i_{S_i-T_i})$,
\begin{multline*}
\frac \dd{\dd t} (V_i(X_{T_i+\tau^i_t})-u_i-Y^i_t)_+^2
= (V_i(X_{T_i+\tau^i_t})-u_i-Y^i_t)_+ \\
\times \Big( -\frac {d-1}{Y^i_t} 
+ \frac{\Delta V_i(X_{T_i+\tau^i_t})}
{|\nabla V_i(X_{T_i+\tau^i_t})|^2} 
 -
\beta_{T_i+\tau^i_t} \frac{\nabla V_i(X_{T_i+\tau^i_t})\cdot 
\nabla U(X_{T_i+\tau^i_t})}{|\nabla V_i(X_{T_i+\tau^i_t})|^2}
+ \frac {\log (1+t/(2\kappa_i)^2)}{c}\Big)\\
\leq  (V_i(X_{T_i+\tau^i_t})-u_i-Y^i_t)_+\Big(\beta_0 - \beta_{T_i+\tau^i_t}+
\frac {\log (1+t/(2\kappa_i)^2)}{c}\Big) \leq 0.
\end{multline*}
For the last inequality, we used that $\beta_{T_i+\tau^i_t}=\beta_0 + c^{-1}\log(1+\tau^i_t)$
and that $\tau^i_t \geq t/(2\kappa_i)^2$ for all $t \in [0,\theta^i_{\zeta-T_i})$, 
because $\theta^i_t \leq (2\kappa_i)^2 t$ for all $t\in [0,{\zeta-T_i})$, recall
that $|\nabla V_i|\leq 2\kappa_i$.
But
$ V_i(X_{T_i})-u_i-Y^i_0=-c\leq 0$, whence indeed, $V_i(X_{T_i+\tau^i_t})
\leq u_i+Y^i_t$ for all $t \in [0,\theta^i_{S_i-T_i})$.

\vip

On $\{T_i<\infty,\;\sup_{t\geq 0}Y^i_t < c \max\{1,\e\kappa_i\}(1+\eta)\}$,
we thus have $V_i(X_{T_i+\tau^i_t})< u_i+c \max\{1,\e\kappa_i\}(1+\eta)$ for all 
$t\in [0,\theta^i_{S_i-T_i})$, so that $V_i(X_{T_i+t})< u_i+c \max\{1,\e\kappa_i\}(1+\eta)$
for all $t\in [0,S_i-T_i)$, whence $S_i=\infty$, because
$V_i(x)=U(x)=v_i\geq u_i+\alpha \max\{1,\e\kappa_i\}>u_i+c \max\{1,\e\kappa_i\}(1+\eta)$ 
for all $x \in \partial \cZ_i^+$.
Hence on $\{T_i<\infty\}$,
$$
\PR(S_i=\infty|\cF_{T_i})\geq \PR\Big(\sup_{t\geq 0}Y^i_t < c\max\{1,\e\kappa_i\}(1+\eta)\Big) 
\geq q(\e,\eta)
$$
by Lemma \ref{troc3}.
\end{proof}

\section{Other proofs}\label{contrexxx}

We first verify that in $(H_1(a))$, the condition $a>c(d-2)/2$ is sharp.

\begin{proof}[Proof of Proposition \ref{contrex}]
We assume here that $\nabla U(x)=\frac{2\alpha x}{(1+|x|^2)(1+\log(1+|x|^2))}$, with 
$0<\alpha<c(d-2)/2$. 

\vip

{\it Step 1.} Considering the $1$-dimensional Brownian motion $W_t=\int_1^{t+1} \frac{X_s}{|X_s|}\cdot \dd B_s$,
which is independent from $X_1$, and denoting $S_t=|X_{t+1}|$, the It\^o formula reads (recall $\beta_0 = 0$)
$$
S_t=S_0+W_t + \frac{d-1}{2}\intot \frac{\dd s}{S_s} 
- \frac \alpha c \intot \frac{S_s \log(2+s)\dd s}{(1+S_s^2)(1+\log(1+S_s^2))}.
$$

{\it Step 2.} We set $\delta=d-2\alpha/c>2$, $\delta'=(\delta+2)/2 \in (2,\delta)$
and consider the Bessel process
$$
R_t = 1+ W_{t}+ \frac{\delta'-1}{2}\intot \frac{\dd s}{R_s}.
$$
By Proposition \ref{rap}-(d), we know that a.s.,
$\liminf_{t\to \infty} t^{-1/2}(\log t)^\nu R_t =\infty$, where $\nu=4/(\delta'-2)$.
Hence
$\limsup_{t\to \infty} [\log(2+t)]/[1+\log(1+R_t^2)] \leq 1$ a.s.
We fix $\eta=(\delta-2)c/(4\alpha)>0$, which gives $(d-1)/2-\alpha(1+\eta)/c
=(\delta'-1)/2$ and
we consider $K\geq  1$ large enough so that $\PR(\Omega_K)\geq 3/4$, where
$$
\Omega_K=\Big\{\hbox{for all } t\geq K, \;\;
R_t \geq 1 \; \hbox{ and } \; \frac{\log(2+t)}{1+\log(1+R_t^2)} \leq 1+\eta
\Big\}.
$$

{\it Step 3.} There is $A>0$ such that,
$\PR(\Omega'_K\ |\ S_0 \geq A)\geq 3/4$, where $\Omega'_K=\{S_{K} \geq R_K+1\}$.

\vip

Indeed, $S_{K}-R_K\geq S_0+Z_K$, where 
$Z_K=W_{K}- \alpha/c \int_0^{K} \log(2+s)\dd s-R_K$ is independent from $S_0$.
As a consequence, $\PR(\Omega'_K\ |\ S_0 \geq A) \geq \PR(Z_K\geq 1-A)$,
that goes to 1 as $A$ goes to infinity.

\vip

{\it Step 4.} We show that $\{S_0\geq A\} \cap \Omega_K \cap \Omega_K'
\subset \{\forall \; t\geq K,\; S_{t} > R_t\}$. This will conclude the proof,
since $\lim_{t\to \infty} R_t=\infty$ a.s. and, $\Omega_K$ being independent from $S_0$, 
\[\PR\left( \{S_0\geq A\} \cap \Omega_K \cap \Omega_K'\right)
\ = \  \PR(\Omega_K \cap \Omega_K'\ | \ S_0 \geq A) \PR(S_0\geq A) \ \geq\  \frac12 \PR(S_0\geq A),\]
which is positive by Girsanov's theorem, whatever the initial condition $x_0$.

\vip

We thus work on $\{S_0\geq A\} \cap  \Omega_K \cap \Omega_K'$ and introduce 
$\tau=\inf\{t \geq K : S_{t} \leq R_t\}$. For all
$t\in [K,\tau)$, we have
$$
\frac \dd{\dd t}(S_{t}-R_t)=\frac{d-1}{2S_{t}} - \frac{\alpha S_{t} \log(2+t)}{c(1+S_{t}^2)
(1+\log(1+S_{t}^2))}
- \frac{\delta'-1}{2R_t}\geq \frac{\delta'-1}{2}\Big(\frac 1{S_{t}}-\frac1{R_t}\Big),
$$
because $S_{t}/(1+S_{t}^2) \leq 1/S_{t}$, because $\log(2+t)/(1+\log(1+S_{t}^2)) \leq
\log(2+t)/(1+\log(1+R_t^2))\leq 1+\eta$ since $t\in [K,\tau)$ and since we work
on $\Omega_K$, and because $(d-1)/2-\alpha(1+\eta)/c=(\delta'-1)/2$.
Hence, still for $t\in [K,\tau)$,
$$
\frac \dd{\dd t}(S_{t}-R_t) \geq - \frac{\delta'-1}{2 R_tS_{t}}(S_{t}-R_t) \geq 
- \frac{\delta'-1}{2}(S_{t}-R_t),
$$
since $S_{t}\geq R_t \geq 1$ by definition of $\tau$ and $\Omega_K$. 
Finally, as $S_{K}-R_K\geq 1$ by definition of $\Omega_K'$, we conclude that
$S_{t}-R_t \geq \exp(-(\delta' - 1) t / 2)$ for all $t\in [K,\tau)$,
and this implies that $\tau=\infty$ as desired. 
\end{proof}

Finally, we give the

\begin{proof}[Proof of Proposition \ref{contrex2}]
We fix $p\geq 1$ and set $u_0=0$, $v_0=1$, and $u_i=\exp^{\circ p}(i)$, 
$v_i=u_i+1$ for $i\geq 1$.
We define the function $g:[0,\infty)\to [0,\infty)$,
continuous and linear by pieces, by $g(u_i)=2i$ and $g(v_i)=2i+3$ for all $i\geq 1$.
We then introduce a smooth version $h$ of $g$, with the very same table of
variations, such that
$h(r)= g(r)$ for all 
$r\in \cup_{i\geq 0}(\{u_i\}\cup[u_i+0.1,v_i-0.1]\cup\{v_i\}\cup[v_i+0.1,u_{i+1}-0.1])$ 
and such that $U:\R \to [0,\infty)$ defined by $U(x)=h(|x|)$ is $C^\infty$. 
Then, $U$ satisfies 
$(A)$ with $c_*=1$. 
It also satisfies $(H_2(2))$, with
$a_i=u_i+0.1$, $b_i=v_i-0.1$, $\delta_i=1/3$. Indeed, $|x| \in [a_i,b_i]$
implies that $\frac x{|x|}\cdot\nabla U(x)= h'(|x|)=3$, and we have 
$b_i-a_i=0.8 \geq 2\delta_i$.
Finally, for all $x \in \R^d$ such that $|x|\geq \exp^{\circ p}(1)$, we have 
$|x| \in [u_i,u_{i+1}]$ with $i=\lfloor \log^{\circ p} |x|\rfloor$, whence
$U(x) \in [2i,2i+3]$. Hence we clearly have $\log^{\circ p} |x|\leq U(x) 
\leq 3\log^{\circ p} |x|$ as soon as $|x|$ is large enough.
\end{proof}






\end{document}